\documentclass[a4paper,12pt]{article}
\usepackage{lscape}
\usepackage{a4,amssymb,amsmath,amsthm,url,lineno,threeparttable}
\usepackage{bezier,amsfonts,amssymb,graphicx,amsthm,url}
\usepackage[english]{babel}
\usepackage{bezier,amsfonts,latexsym,amsmath,amssymb,graphicx,color}

\title{A note on totally-omnitonal graphs}
\date{}
\begin{document}
\newtheorem{theoremnum}{Theorem}[section]
\newtheorem*{lemma*}{Lemma}
\newtheorem{theorem}{Theorem}
\renewcommand*{\thetheorem}{\Alph{theorem}}

\newtheorem{definition}{Definition}[section]
\newtheorem{observation}[theoremnum]{Observation}
\newtheorem{proposition}[theorem]{Proposition}
\newtheorem{corollary}[theoremnum]{Corollary}
\newtheorem{lemma}[theoremnum]{Lemma}
\newcommand{\bal}{{\rm bal}}
\newcommand{\sbal}{{\rm sbal}}
\newcommand{\ot}{{\rm ot}}
\newcommand{\ex}{{\rm ex}}
\newtheorem{Ex}[theoremnum]{$\rhd$ Example}
\newcommand{\tot}{{\rm tot}}

\DeclareGraphicsExtensions{.pdf,.png,.jpg}

\author{Yair Caro \\ Department of Mathematics\\ University of Haifa-Oranim \\ Israel \and Josef  Lauri\\ Department of Mathematics \\ University of Malta
\\ Malta \and Christina Zarb \\Department of Mathematics \\University of Malta \\Malta }

\maketitle

\begin{abstract}
Let the edges of the complete graph $K_n$ be coloured red or blue, and let $G$ be a graph with $|V(G)| < n$.  Then $\ot(n,G)$ is defined to be the minimum integer, if it exists, such that  any  such colouring of $K_n$ contains a copy of $G$ with $r$ red edges and $b$ blue edges for any $r,b \geq 0$ with $r+b= e(G)$.   If $\ot(n,G)$ exists for every sufficiently large $n$, we say that $G$ is \emph{omnitonal}.   Omnitonal graphs were introduced by Caro, Hansberg and Montejano  [arXiv:1810.12375,2019].  Now let $G_1$, $G_2$ be two copies of $G$ with their edges coloured red or blue.  If there is a colour-preserving isomorphism from $G_1$ to $G_2$ we say that the 2-colourings of $G$ are equivalent.  Now we define $\tot(n,G)$  to be the minimum integer, if it exists, such that  any  such colouring of $K_n$ contains all non-quivalent colourings of $G$ with $r$ red edges and $b$ blue edges for any $r,b \geq 0$ with $r+b= e(G)$.   If $\tot(n, G)$ exists for every sufficiently large $n$, we say that G is \emph{totally-omnitotal}.

In this note we show that the only totally-omnitonal graphs are stars or star forests namely a forest all of whose components are stars.
\end{abstract}
\section{Introduction}

By a 2-colouring of the complet graph $K_n$ we mean a function $f:E(K_n) \to \{red,blue\}$.  The set of edges of $K_n$ coloured red or blue is denoted by $R$ or $B$, respectively.  For short we also denote by $R$, $B$ the subgraphs of $K_n$ induced by these edge sets.  If $G$ is a subgraphs of such a 2-coloured $K_n$ with $r$ red edges $b$ blue edges we say that $K_n$ contains an $(r,b)$-coloured copy of $G$.
We recall the definition of omnitonal graphs from \cite{2018arXivCHM}. For a given graph $G$, $\ot(n,G)$ is defined to be the minimum integer, if it exists, such that  any 2-colouring of $K_n$, $n \geq |V(G)|$,  with $\min\{|R|,|B|\} > \ot(n,G)$ contains an $(r,b)$-coloured copy of $G$ for any $r \geq 0$ and $b \geq 0$ such that $r+b= e(G)$, where $e(G)=|E(G)|$.  If $\ot(n,G)$ exists for every sufficiently large $n$, we say that $G$ is \emph{omnitonal}.   

We now define totally-omnitonal graphs .  Let $G_1$, $G_2$ be two $(r,b)$-coloured copies of $G$.  Then if there is a colour preserving  isomorphism $\phi: G_1 \to G_2$, we say that the two colourings $G_1$ and $G_2$ of $G$ are equivalent.  Otherwise the colourings are said to be non-equivalent.  Now, for a given graph $G$, $\tot(n,G)$ is defined to be the minimum integer, if it exists, such that  any 2-colouring of $E(K_n)$ with $\min\{|R|,|B|\} > \tot(n,G)$ contains every non-equivalent $(r,b)$-coloured copy of $G$ for any $r \geq 0$ and $b \geq 0$ such that $r+b= e(G)$.  If $\tot(n,G)$ exists for every sufficiently large $n$, we say that $G$ is  \emph{totally-omnitonal}.

For other graph-theoterical terms we refer the reader to West \cite{west2017introduction}. We just recall that a star, denoted by $K_{1,p}$ is the graph consisting of one vertex joined to each of $p$ other vertices. Therefore $K_2$ is the star $K_{1,1}$.

The main aim of this note is to show that a graph $G$ is totally-omnitonal if and only  if it is a star or a star forest, namely a forest all of whose components are stars.

In several places in this note we shall make use of the following result, which is part of Theorem 4.1 in \cite{2018arXivCHM}.

\begin{theorem}  \label{chm4_1}
Let $n$ and $k$ be positive integers such that $n \geq 4k$. Then 
\[
\ot(n, K_{1,k})= 
\begin{cases}
\lfloor (\frac{(k-1}{2})n \rfloor, & \text{for $k \leq 3$},\\
(k-2)n - \frac{k^2}{2}+ \frac{3}{2}k-1, & \text{for $k\geq 4$}.
\end{cases}
\]
\end{theorem}

\section{Results}

\subsection{A canonical colouring for $K_n$}

We shall need a colouring of this type for the edges of $K_n$.  Let $A$ be a non-empty subset of $V(K_n)$ of size at most $n-1$, and let $B$ be $V(K_n)-A$. Colour red all $\binom{|A|}{2}$ edges joining vertices in $A$ and colour blue all the remaining $\binom{|B|}{2} + |A|\cdot|B|$ edges of $K_n$. We observe that in such a colouring of $K_n$, there is no path $P_4$ with an $r-b-r$ colouring, that is, a colouring of the edges of $P_4$ such that the middle edge is coloured blue and the pendant edges are both coloured red. Also, there is no colouring of $K_3$ with two red and one blue edges.  In order to use such a 2-colouring of $K_n$ to show that $P_4$  and $K_3$ are not totally-omnitonal we shall need to show, for reasons which shall become clear below, that there is such a colouring for an infinite sequence of complete graphs in which the number of red edges is equal to the number of blue edges. 

For this to hold, suppose $|A|=r$.   We shall now use an argument employed in \cite{caro2019zero2}.  We require that 
\[\frac{r(r-1)}{2} = \frac{n(n-1)}{4},   \]
that is,
\[2(r^2-r)=n^2-n.\]
But this is equivalent to 
\[ (2n-1)^2 - 2(2r-1)^2 = -1,\]
and, if we let $y$ and $x$ be, respectively, the two odd integers $2n-1$ and $2r-1$, we obtain,
\[ y^2 - 2x^2 = -1.\]
But this is Pell's equation which is known to have an infinite number of solutions for $x$ and $y$ \cite{andreescu}.

So we define a \emph{canonical colouring} of $K_n$, if it exists, to be a colouring in which the edges of a subclique are coloured red, while all the other edges are coloured blue, and the number of red edges is equal to the number of blue edges.  We therefore have the following result from \cite{caro2019zero2}.

\begin{theoremnum}
	There is an infinite sequence of complete graphs $K_n$ for which a canonical colouring exists.
\end{theoremnum}

\subsection{$P_4$ and $K_3$ are not totally-omnitotal}

Suppose that $P_4$ (or $K_3$) is totally-omnitonal.
Therefore exists a $t(n)= \tot(n, P_4)$ (or $t(n)=\tot(n, K_3)$), such that  any two-colouring of $K_n$  with $\min\{|R|,|B|\} > t(n)$ contains an $r-b-r$ colouring of $P_4$ (or, a $(2,1)$-colouring of $K_3$)  for $n$ sufficiently large. Note that $t(n)$ must be less than $n(n-1)/4$. However, we have seen that for any $N$ there is a $K_n$ with $n>N$ which has a canonical colouring. In this colouring, $|R|=|B|=n(n-1)/4$, and $K_n$ does not contain an $r-b-r$ colouring of $P_4$ (nor, a $(2,1)$-colouring of $K_3$). We have therefore proved the following.

\begin{lemma} \label{p4k3}
	Both $P_4$ and $K_3$ are not totally-omnitonal.
\end{lemma}

We now use this lemma to characterise connected omnitonal graphs.

\begin{theoremnum} \label{not_tot}
	Let $G$ be a connected graph on at least three vertices which is not a star.  Then $G$ is not totally-omnitonal.
\end{theoremnum}

\begin{proof}
	If $G$ is not a star or $K_3$ which we already proved to be non-omnitonal,  then it is well-known that it must have a pair of independent edges. Let $e_1, e_2$ be the closest pair of independent edges. They must therefore be joined by an edge $e_3$. Now colour $e_3$ blue and all the other edges of $G$ red. 

This is a specific $(e(G)|– 1, 1)$-colouring of $G$, but as we have shown in Lemma \ref{p4k3}, for infinitely many values of $n$ there is a canonical colouring of $K_n$  which  does not contain this colouring of $G$ due to the specific colouring of the $P_4$ subgraph of $G$. Therefore $G$ cannot be totally-omnitotal.
\end{proof}

However we do have the following. 

\begin{lemma}
	Stars are totally-omnitonal.
\end{lemma}

\begin{proof}
	
	It is known from Theorem \ref{chm4_1} in \cite{2018arXivCHM} stated above, that $K_{1,p}$ is omnitonal. Therefore there is a number $\ot(n,K_{1,p})$ such that, for any positive integers $r, b$ with $r+b=p$, and for all $n$ sufficiently large, if $E(K_n)$ is two-coloured with $\min\{|R|, |B|\}>\ot(n,K_{1,p})$, then it contains an $(r,b)$-coloured copy of $K_{1,p}$. However, by the symmetry of the edges of $K_{1,p}$, any two $(r,b)$-coloured copies of $K_{1,p}$ with $r+b=p$ are equivalent. Therefore $K_{1,p}$ is totally-omnitonal with $\tot(n,K_{1,p})=\ot(n,K_{1,p})$.
\end{proof}

From all the above we can conclude the following.

\begin{corollary}
	A connected graph $G$ is totally-omnitotal if and only if it is a star.
\end{corollary}

\subsubsection{Disconnected totally-omnitonal graphs}

And so we come to our main theorem. By a star-forest we shall mean a graph all of whose connected components are stars.
 
\begin{theoremnum}
	A graph $G$ is totally omnitonal if and only if it is a star forest.
\end{theoremnum}

\begin{proof}
	If $G$ is connected then we are done since we have already shown that the only connected totally-omnitonal graphs are stars. Therefore suppose $G$ is disconnected.
	 
	If even one component of $G$ is not totally-omnitonal, then $G$ is not totally-omnitotal. Therefore if even one component of $G$ is not a star then $G$ is not totally-omnitonal. This is because:
\begin{enumerate}
\item{If there is a $K_3$ component  then we colour $E(K_3)$ with one blue edge and two blue edges, and all the other edges of $G$  are coloured red. This is a specific $(e(G)-1,1)$-coloured pattern of $G$, which requires that $K_3$ will be coloured in a $(2,1)$-colouring, which is impossible as we have shown in Lemma  \ref{p4k3}.}
\item{If there is no $K_3$ but there is a component which is not a star, then this component must contain $P_ 4$,  and we colour the middle edge of this $P_4$ blue and all the other edges of $G$ red.  This is a specific $( e(G)-1 ,1)$–coloured pattern of $G$, but we have shown the canonical coloring cannot contain any $r-b-r$ coloured $P_4$,  hence all components must be stars.}
\end{enumerate}

Conversely, suppose each component of $G$ is totally-omnitonal, that is, each component is a star. We need to show that $G$ is totally-omnitonal.

We know from Theorem \ref{chm4_1} that  $\ot(n,K_{1,k })  <(k-1)n$   for $n \geq 4k$.  Let $G =  \cup K_{1,p_ j} for j = 1,\ldots,q$   and $p_1 \geq p_2 \geq \ldots \geq   p_ q$.  

Let  $ n\geq 4(p_1+p_2+ \ldots +p_q +q-1)$ then  we claim that $\tot(n,G)  < M(p_1,\ldots,p_q,n) :=  (p_1+p_2+..p_q +q-2 )n$.  We have to show first that  the conditions on $n$ and $M$ are feasible for every $q \geq 1$, namely  that for $n\geq 4(p_1+p_2+ \ldots +p_q +q-1)$  a colouring with $\min \{|R|,|B|\} \geq M(p_1,\ldots,p_q,n)$    exists.  Hence we have to show that $2(p_1+\ldots+p_q +q-2)n \leq n(n-1)/2  $.  But this is equivalent to   $4(p_1+\ldots+p_q +q -2) =  4(p_1 + \ldots+p_q +q -1)  - 4 \leq  n-4  < n-1$.  Hence the condition is satisfied with $n \geq 4(p_1+\ldots+p_q +q-1)$ and $M(p_1,\ldots,p_q,n)$ as defined above.

	We now prove the theorem by induction on $q$.  For $q = 1$ the result is true by Theorem \ref{chm4_1}.  Assume  result is true for $q-1$ components and assume $G$ has $q$ components.

Let $f$ be any colouring of the edges of $G$. We wish  to show that in any 2-colouring  $g :  E(K_n) \to \{ red ,blue \}$ with  $\min \{|R|,|B|\} \geq  (p_1+\ldots+p_q +q - 2)n$ there is a copy of $G$ on which  $g$ restricted to the edges of $G$ is equivalent to the colouring $f$ of $G$.

We have already shown above that such a colouring $g$ for $K_n$ exists for $n \geq 4( p_1+p_2+ \ldots+ p_ q +q -1)$.  So suppose $n \geq 4( p_1+p_2+ \ldots+ p_q +q-1)$. Let $G^* = G \backslash K_{1,p_q}$.  Since  $p_ 1 \geq p_q  $, $n  > 4p_q$   and $\min \{|R|,|B|\}  >  n(p_q-1)$, it follows from Theorem  \ref{chm4_1}  that there is a copy of $K_{1,p_q}$  which is precisely $f$-coloured, namely the colouring of $K_{1,p_q}$ is equivalent to the colouring induced by $f$.

Remove  the vertices of this $f$-coloured $K_{1,p_q}$ from $V(K_n)$ and remove the edges incident with at least one vertex of $V(K_{1,p_q})$.  We are left with  $n^*$  vertices  where $n^* = n - (p_q+1) \geq  4(p_1+\ldots+p_{q-1}+p_q +q -1)  - (p_q + 1)  >    4(p_1+\ldots+p_{q-1}+ q-2)$ and with  $\min \{|R|,|B|\} \geq  (p_1+\ldots+p_q +q – 2 )n  -  \binom{p_q  +1}{2}  -  ( p_q +1 )(n – p_q -1) >   (p_1+\ldots +p_q +q -2)n  -n(p_q +1) =  (p_1+\ldots+p_{q-1} +q-3)n > ( p_1+\ldots+p_{q-1} + q-3)n^*$. 

So the conditions for $q – 1$ components are satisfied.

Therefore by the induction hypothesis,  $K_{n^* }$ contains  an $f$-coloured copy of $G^*$ which together with the deleted $f$-coloured copy of $K_{1,p_q}$,  giving the required $f$-coloured copy of $G$ in $E(K_n)$.
\end{proof}

The bounds we have proved can now be stated as a corollary.

\begin{corollary}
Let $G$ be a star-forest with components $K_{1,p_1} \cup K_{1,p_2} \cup \ldots \cup K_{1,p_q}$ with $p_1 \geq p_2 \geq \ldots \geq p_q$.  Then, for  $n \geq 4(p_1+\ldots+p_q +q-1)$, any 2-colouring of $E(K_n)$ with $\min \{|R|,|B|\} > (p_1+ \ldots+ p_q +q-2)n$ contains every non-equivalent $(r,b)$-coloured copy of $G$, for any $r,b \geq 0$ and $r+b=e(G)$.
\end{corollary}

\bibliographystyle{plain}
\bibliography{mbmbib}
\end{document}